\documentclass[a4paper,10pt]{amsart}
\usepackage[centertags]{amsmath}
\usepackage[english]{babel}
\usepackage{amsfonts}
\usepackage{amssymb}
\usepackage{amsthm}
\usepackage[usenames,dvipsnames]{color}
\usepackage{newlfont}
\usepackage{graphicx}
\usepackage{epstopdf}
\usepackage{cite}
\usepackage{bm}
\usepackage{blkarray}
\usepackage{multirow}
\usepackage{multicol}
\usepackage[all,cmtip]{xy}
\usepackage{xypic}
\usepackage{longtable}

\newcommand{\xdownarrow}[1]{%
    {\left\downarrow\vbox to #1{}\right.\kern-\nulldelimiterspace}}

\newcommand{\Z}{\mathbb{Z}}

\newcommand{\bq }{\begin{equation}}
\newcommand{\eq }{\end{equation}}
\theoremstyle{plain}
\newtheorem{thm}{Theorem}[section]
\newtheorem{lem}[thm]{Lemma}

\newtheorem{prop}[thm]{Proposition}

\newtheorem{cor}[thm]{Corollary}
\newtheorem{rem}[thm]{Remark}

\theoremstyle{definition}
\newtheorem{defn}[thm]{Definition}

\DeclareMathOperator{\modu}{{{mod}}}
\theoremstyle{example}

\title{The Picard group of Brauer-Severi varieties}

\author[E. Badr] {Eslam Badr}
\address{$\bullet$\,\,Eslam Essam Ebrahim Farag Badr}
\address{Departament Matem\`atiques, Edif. C, Universitat Aut\`onoma de Barcelona\\
08193 Bellaterra, Catalonia, Spain} \email{eslam@mat.uab.cat}
\address{Department of Mathematics,
Faculty of Science, Cairo University, Giza-Egypt}
\email{eslam@sci.cu.edu.eg}
\author[F. Bars] {Francesc Bars}
\address{$\bullet$\,\,Francesc Bars}
\address{Departament Matem\`atiques, Edif. C, Universitat Aut\`onoma de Barcelona\\
08193 Bellaterra, Catalonia} \email{francesc@mat.uab.cat}

\author[E. Lorenzo]{Elisa Lorenzo Garc\'ia}
\address{$\bullet$\,\,Elisa Lorenzo Garc\'ia}
\address{IRMAR - Universit\'e de Rennes 1\\
35042 Rennes Cedex, France}
\email{elisa.lorenzogarcia@univ-rennes1.fr}

\thanks{E. Badr and F. Bars are supported by MTM2016-75980-P}

\keywords{}
\subjclass[2010]{}

\begin{document}

\maketitle

\begin{abstract} In this note we provide explicit generators of the Picard groups of cyclic Brauer-Severi varieties defined over the base field. In particular, for all Brauer-Severi surfaces.
To produce these generators we use the Twisting Theory for smooth
plane curves.
\end{abstract}

\section{Introduction}
Let $B/k$ be a Brauer-Severi variety over a perfect field $k$, that
is, a projective variety of dimension $n$ isomorphic over
$\overline{k}$ to $\mathbb{P}^n_k$. Its group the Picard
$\operatorname{Pic}(B)$ it is known to be isomorphic to
$\mathbb{Z}$.
  As far as we know, the first explicit equations defining a non-trivial Brauer-Severi surface in the literature are in \cite{BaBaEl1}. After this,
   an algorithm to compute these equations
  for any Brauer-Severi variety is given in \cite{Lo2}. In the appendix, we explain an alternative way to compute them {for the case of dimension $2$ and} using twists of smooth plane curves.

  In this note, we show an explicit and concrete generator of the Picard group of any Brauer-Severi variety corresponding to a cyclic algebra
  in its class inside the Brauer group $\operatorname{Br}(k)$ of $k$. More precisely, for Brauer-Severi surfaces we obtain the folowwing result.

\begin{thm}\label{thm-main} Let $B$ be the Brauer-Severi surface corresponding to a cyclic algebra $(L/k,\chi,a)$ of dimension $3^2$ as in Theorem \ref{thm-cyclic}. A smooth model of $B$ inside $\mathbb{P}^9_k$ is given by  the intersection  $\cap_{\sigma\in\operatorname{Gal}(L/k)}\,^{\sigma}X$ where $X/L$ is the variety in $\mathbb{P}^9$ defined by the set of equations:
$$
    \begin{array}{c}
    a^2(l_1\omega_0+l_2\omega_6+l_3\omega_9)(l_3\omega_0+l_1\omega_6+l_2\omega_9)^2=(l_3\omega_1+l_1\omega_5+l_2\omega_7)^3\\
   a(l_1\omega_1+l_2\omega_5+l_3\omega_7)(l_3\omega_0+l_1\omega_6+l_2\omega_9)^2=(l_3\omega_1+l_1\omega_5+l_2\omega_7)^2(l_3\omega_2+l_1\omega_3+l_2\omega_8)\\
    a(l_1\omega_2+l_2\omega_3+l_3\omega_8)(l_3\omega_0+l_1\omega_6+l_2\omega_9)^2=(l_3\omega_1+l_1\omega_5+l_2\omega_7)^2(l_3\omega_0+l_1\omega_6+l_2\omega_9)\\
   a(l_2\omega_2+l_3\omega_3+l_1\omega_8)(l_3\omega_0+l_1\omega_6+l_2\omega_9)^2=(l_3\omega_1+l_1\omega_5+l_2\omega_7)(l_3\omega_2+l_1\omega_3+l_2\omega_8)^2\\
    \omega_4(l_3\omega_0+l_1\omega_6+l_2\omega_9)^2=(l_3\omega_1+l_1\omega_5+l_2\omega_7)(l_3\omega_2+l_1\omega_3+l_2\omega_8)(l_3\omega_0+l_1\omega_6+l_2\omega_9)\\
    a(l_2\omega_0+l_3\omega_6+l_1\omega_9)(l_3\omega_0+l_1\omega_6+l_2\omega_9)^2=(l_3\omega_2+l_1\omega_3+l_2\omega_8)^3\\
    (l_2\omega_1+l_3\omega_5+l_1\omega_7)(l_3\omega_0+l_1\omega_6+l_2\omega_9)^2=(l_3\omega_2+l_1\omega_3+l_2\omega_8)^3(l_3\omega_0+l_1\omega_6+l_2\omega_9),
    \end{array}
$$
where $\{l_1,l_2,l_3\}$ is a non-zero trace normal basis of $L$. Its Picard group $\operatorname{Pic}(B)$ is generated by the hyperplane
$$ \omega_0+\omega_6+\omega_9=0,$$
    which is a genus 1 curve with $k$ as a field of definition.
More generally, for a positive element
$d'\in\mathbb{Z}\simeq\operatorname{Pic}(B)$, we have a generator of
$d'\operatorname{Pic}(B)$ given by
   $$
    (l_1\omega_0+l_2\omega_6+l_3\omega_9)^{d'}+(l_2\omega_0+l_3\omega_6+l_1\omega_9)^{d'}
    +(l_3\omega_0+l_1\omega_6+l_2\omega_9)^{d'}=0,
    $$
    and defining a curve of genus $\frac{(3d'-1)(3d'-2)}{2}$ defined over $k$.
\end{thm}

\section{Brauer-Severi varieties}

\begin{defn}\label{twistsefn}
Let $V$ be a smooth quasi-projective variety over $k$. A variety
$V'$ defined over $k$ is called a twist of $V$ over $k$ if there is
a $\overline{k}$-isomorphism
$V'\otimes_k\overline{k}{\stackrel{\Phi}{\longrightarrow}}\overline{V}:=V\otimes_k\overline{k}.$
The set of all twists of $V$ modulo $k$-isomorphisms is denoted by
$\text{Twist}_k(V)$, whereas the set of all twists $V'$ of $V$ over
$k$, such that $V'\times_kK$ is $K$-isomorphic to $V\times_kK$ is
denoted by $\operatorname{Twist}(V,K/k)$.
\end{defn}

\begin{thm}\cite[Ch.III, \S 1.3]{Se}
Following the above notations, for any Galois extension $K/k$, there
exists a bijection
\begin{eqnarray*}
\theta&:&\operatorname{Twist}(V,K/k)\rightarrow\operatorname{H}^1(\operatorname{Gal}(K/k),\operatorname{Aut}_K(V\times_kK))\\
&&V'\times_k K{\stackrel{\Phi}{{\cong}}}V\times_k K\mapsto
\xi(\tau):=\Phi\circ\,^{\tau}\Phi^{-1}
\end{eqnarray*}
 where $\operatorname{Aut}_K(.)$ denotes the group of $K$-automorphisms of the object over $K$.
\end{thm}

For $K=\overline{k}$, the right hand side will be denoted by $\operatorname{H}^1(k,\operatorname{Aut}_{\overline{k}}(\overline{V}))$ or simply $\operatorname{H}^1(k,\operatorname{Aut}(\overline{V}))$.

\begin{defn}
A Brauer-Severi variety $B$ over $k$ of dimension $n$ is a twist of $\mathbb{P}^n_k$. The set of all isomorphism classes of Brauer-Severi varieties of dimension $r$ over $k$ is denoted by $\operatorname{BS}^k_r$.
\end{defn}

\begin{cor}(cf. J. Jahnel \cite[Corollary
4.7]{Ja})\label{bijectionBS}\label{cor-Jah}
The set $\operatorname{BS}^k_n$ is in bijection with $\operatorname{Twist}_k(\mathbb{P}^n_k)=\operatorname{H}^1(k,\operatorname{PGL}_{n+1}(\overline{k}))%=\operatorname{H}^1(\operatorname{Gal}(k^{sep}/k),\Aut_{k^{sep}}(\mathbb{P}^n_{k^{sep}}))$.
$.
\end{cor}

\subsection{Brauer-Severi surfaces}

Let $L/k$ be a Galois cyclic cubic extension, with $\text{Gal}(L/k)=\langle\sigma\rangle$. Fix a character 
$$\chi:\text{Gal}(L/k)\rightarrow\Z/3\Z,$$ which is equivalent to fix $\sigma'$ a generator of $Gal(L/k)$ such that $\chi(\sigma')=1$.
Given $a\in k^*$, we consider a $k$-algebra $(L/k,\chi,a)$ as
follows: As an additive group, $(L/k,\chi,a)$ is an $3$-dimensional
vector space over $L$ with basis $1,e,e^2$: $(L/k,\chi,a):=L\oplus
Le\oplus Le^2.$ Multiplication is given by the relations:
$e\,.\,\lambda=\sigma'(\lambda)\,.\,e$ for $\lambda\in L$, and
$e^3=a$. The algebra $(L/k,\chi,a)$ is called the \emph{cyclic
algebra} associated to the character $\chi$ and the element $a\in
k$.

\begin{thm}\label{thm-cyclic}
Any non-trivial Brauer-Severi surface $B$ over $k$ corresponds, modulo $k$-isomorphism, to a cyclic algebra $(L/k,\chi,a)$ of dimension $9$, for some Galois cubic extension $L/k$ and $a\in k^*$, which is not a norm of an element of $L$. If $\operatorname{Gal}(L/k)=\langle\sigma\rangle$, then the image of $B$ in $\operatorname{H}^1(k,\operatorname{PGL}_3(\overline{k}))$ is given by
$$
\xi(\sigma)=
         \begin{pmatrix}
           0 & 0 & a \\
           1 & 0 & 0 \\
           0 & 1 & 0
           \end{pmatrix}
$$
where $\operatorname{Gal}(L/k)=<\sigma>$ and $\chi(\sigma)=-1\modu 3$. Moreover, the Brauer-Severi surface attached to $(L/k,\chi,a)\in\operatorname{Az}_3^k$ is trivial if
and only if $a$ is the norm of an element of $L$
\end{thm}

This Theorem is a conclusion from the fact that
$\operatorname{H}^1(k,\operatorname{PGL}_n(\overline{k}))$ is in
correspondence with the set $\operatorname{Az}^n_k$ of central
simple algebras of dimension $n^2$ over $k${, modulo
$k$-isomorphisms} \cite[chap.X.5]{SeL}, the fact that
$\operatorname{Az}^3_k$ contains only cyclic algebras
\cite{Wed} and the description of cyclic central simple algebras given
in \cite[Example 5.5]{Ten}. For the last statement see \cite[\S
2.1]{Han}.

%
%In \cite{Lo2} an algorithm to compute equations of Brauer-Severi varieties is given.
%In particular, in Theorem $5.2$ a description of the equations for any Brauer-Severi surfaces is given.
%This implies the first part of Theorem \ref{thm-main}.
%This could appear in the proof of Theorem 1.1.

\section{Smooth plane curves}
Fix an algebraic closure $\overline{k}$ of a perfect field $k$. By a
smooth plane curve $C$ over $k$ of degree $d\geq3$, we mean a curve
$C/k$, which is $k$-isomorphic to the zero-locus of a homogenous polynomial equation
$F_C(X,Y,Z)=0$ in $\mathbb{P}^2_k$ without singularities of degree
$d$. In that case, the geometric genus of $C$ is
$g=\frac{1}{2}(d-1)(d-2)$. Assuming that $d\geq 4$, the base
extension $C\times_k\overline{k}$ admits a unique $g^2_d$-linear
system up to conjugation in
$\operatorname{Aut}(\mathbb{P}^2_{\overline{k}})=\operatorname{PGL}_3(\overline{k})$.
It induces a unique embedding $\Upsilon:\overline{C}\rightarrow
\mathbb{P}^2_{\overline{k}}$, up to
$\operatorname{PGL}_3(\overline{k})$-conjugation giving a
$\operatorname{Gal}(\overline{k}/k)$-equivariant map
$\operatorname{Aut}(\overline{C})\hookrightarrow
\operatorname{PGL}_3(\overline{k})$ if $d\geq 4$.

\begin{thm}\label{thm-XR} (Ro\'e-Xarles,\cite{RoXa}) Let $C$ be a curve over $k$ such that $\overline{C}=C\times_k\overline{k}$ is
a smooth plane curve over $\overline{k}$ of degree $d\geq 4$. Let
$\Upsilon:\overline{C}\hookrightarrow \mathbb{P}^2_{\overline{k}}$
be a morphism given by (the unique) $g^2_d$-linear system over
$\overline{k}$, then there exists a Brauer-Severi variety $B$ (of
dimension two) defined over $k$, together with a $k$-morphism
$f:C\hookrightarrow B$ such that
$f\otimes_k\overline{k}:\overline{C}\rightarrow\mathbb{P}^2_{\overline{k}}$
is equal to $\Upsilon$.
\end{thm}

In \cite{BaBaEl1} we constructed twists of smooth plane curves over
$k$ not having smooth plane model over $k$. These twists happened to
be contained in Brauer-Severi surfaces as in Theorem \ref{thm-XR}.

\begin{thm}\label{thm-Sigma}[Theorem 3.1 in \cite{BaBaEl1}] Given a smooth plane curve over $k$: $C/k\subseteq\mathbb{P}^2$ with
degree $d\geq 4$, there exists a natural map
$$
\Sigma:\,\operatorname{H}^1(k,\operatorname{Aut}(C))\rightarrow\operatorname{H}^1(k,\operatorname{PGL}_3(\overline{k})),
$$
whose
%\footnote{Francesc:It is a map of sets, not have additional
%structure in general}
$\Sigma^{-1}([\mathbb{P}^2_k])$ is the set of
twists of $C$ admitting a smooth plane model over $k$, where
$[\mathbb{P}^2_k]$ is the trivial class associated to the trivial
Brauer-Severi surface of the projective {plane} over $k$.

\end{thm}

\begin{rem}[Remark 3.2 in \cite{BaBaEl1}]\label{rem-Sigma} We can reinterpret the map $\Sigma$ in Theorem \ref{thm-Sigma} as the map that sends a twist $C'$ to the Brauer- Severi variety $B$ in Theorem \ref{thm-XR}.
\end{rem}

These results suggest the opposite question, instead of given the curve $C$ and the twist $C'$ and finding the Brauer-Severi surface $B$,
fixing the Brauer-Severi surface $B$ and trying to find the right curve $C$ and the right twist $C'$ to find the $k$-morphism $f:\,C'\hookrightarrow B$.

The main idea is looking for smooth plane curves $C$ of degree a
multiple of $3$, otherwise all their twists are smooth plane curves
over $k$, see Theorem $2.6$ in \cite{BaBaEl1}, and having an
automorphism of the form $[aZ:X:Y]$ to define the twist $C'$ given
by the cocycle that sends a generator $\sigma$ of the degree $3$
cyclic extension $L/k$ defining $B$ to the automorphism $[aZ:X:Y]$.

\begin{lem}\label{lemma-curves}
    For any $a\in {k}^*$ and $d'\in\mathbb{Z}_{\geq1}$, the equation $X^{3d'}+a^{d'}Y^{3d'}+a^{2d'}Z^{3d'}=0,$
    defines a smooth plane curve $C^{d'}_{a}$ over $k$ of degree $3d'$, such that $[aZ:X:Y]$ is an automorphism.
\end{lem}

\section{The Picard group}
\begin{thm}(Lichtenbaum, see \cite[Theorem 5.4.10]{GS})\label{imprtiantexactsequ}
Let $B$ be a Brauer-Severi variety over $k$. Then, there is an exact sequence
\begin{eqnarray*}
0\longrightarrow\operatorname{Pic}(B)\longrightarrow\operatorname{Pic}(B\otimes_k\overline{k}){\stackrel{\operatorname{deg}}{\cong}}\Z{\stackrel{\delta}{\longrightarrow}}\mathrm{Br}(k).
\end{eqnarray*}
The map $\delta$ sends $1$ to the Brauer class corresponding to $B$.
\end{thm}

\begin{thm}\label{prop-generator} Let $B$ be a non-trivial Brauer-Severi surface over $k$, associated to a cyclic algebra $(L/k,\chi,a)$
of dimension $9$ by Theorem \ref{thm-cyclic}. For any integer
$d'\geq 1$, there is a twist $C'$ over $k$ of the smooth plane curve
$C_{d',a}$, living inside $B$ and also defines a generator of
$d'\operatorname{Pic}(B)$.
\end{thm}
\begin{proof}
The twist $C'$ of $C^{d'}_{a}$ given by the inflation map of the
cocycle
$$
\xi(\sigma)=
         \begin{pmatrix}
           0 & 0 & a \\
           1 & 0 & 0 \\
           0 & 1 & 0
           \end{pmatrix}
\in
H^1(\operatorname{Gal}(L/k),\operatorname{Aut}(\overline{C^{d'}_{a}}))$$
as in Theorem \ref{thm-cyclic} lives inside $B$ for any integer
$d'\geq2$ by using Theorem \ref{thm-Sigma} with Remark
\ref{rem-Sigma}. {For $d'=1$, set
$\operatorname{Aut_{\operatorname{L}}}(C^{1}_{a})$ for the subgroup
of automorphisms of $C^{1}_{a}$ acting linearly on the variables $X,Y,Z$. Therefore, the inclusions
$\operatorname{Aut_{L}}(C^{1}_{a})\leq\operatorname{PGL}_3(\overline{k})$
and $\operatorname{Aut_{L}}(C^{1}_{a})\leq
\operatorname{Aut}(\overline{C})$ give us the two natural maps
$\operatorname{\operatorname{L}}:\operatorname{H}^1(k,\operatorname{Aut_{\operatorname{L}}}(C^{1}_{a}))
\rightarrow\operatorname{H}^1(k,\operatorname{Aut}(\overline{C^{1}_{a}}))$,
and
$\Sigma:\operatorname{H}^1(k,\operatorname{Aut_{\operatorname{L}}}(C^{1}_{a}))\rightarrow\operatorname{H}^1(k,\operatorname{PGL}_3(\overline{k}))$
respectively. Second, compose with the $3$-Vernoese embedding
$\operatorname{Ver}_3:\mathbb{P}^2_k\rightarrow\mathbb{P}^9_k$, to
obtain a model of $C$ inside the trivial Brauer-Severi surface
$\operatorname{Ver}_3(\mathbb{P}^2_k)$. Because the image of any 1-cocycle by the
map
$\tilde{\operatorname{Ver}}_3:\operatorname{H}^1(k,\operatorname{PGL}_3(\overline{k}))\rightarrow
\operatorname{H}^1(k,\operatorname{PGL}_{10}(\overline{k}))$, is 
equivalent to a 1-cocycle with values in
$\operatorname{GL}_{10}(\overline{k})$ \cite{Lo2}, and
$\operatorname{H}^1(k,\operatorname{GL}_{10}(\overline{k}))=1$, then
$\tilde{\operatorname{Ver}}_3([B])$ is given in
$\operatorname{H}^1(k,\operatorname{PGL}_{10}(\overline{k}))$ by
$\tau\in\operatorname{Gal}(\overline{k}/k)\mapsto M\circ\,^
{\tau}M^{-1}$, for some $M\in\operatorname{GL}_{10}(\overline{k})$.
Consequently, $(M\circ\operatorname{Ver}_3)(\mathbb{P}^2_k)$ is a
model of $B$ in $\mathbb{P}^9_k$, containing
$(M\circ\operatorname{Ver}_3)(C^{1}_{a})$ inside, which is a twist of
$C^{1}_{a}$ over $k$ associated to $\xi$ by Theorem \ref{thm-cyclic}.}

On the other hand, by Wedderburn \cite{Wed} and Theorem
\ref{imprtiantexactsequ}, the map $\delta$ sends $1$ to the Brauer
class $[B]$ of $B$ inside the 3-torsion $Br(k)[3]$ of the Brauer
group $Br(k)$ of the field $k$. Hence $[B]$ has exact order $3$,
being non-trivial, and so $\operatorname{Pic}(B)$ inside
$\operatorname{Pic}(B\otimes_k\overline{k}=\mathbb{P}^2_{\overline{k}}){\stackrel{\operatorname{deg}}{\cong}}\Z$
is isomorphic to $3\mathbb{Z}$. Moreover,
$C'\times_k\overline{k}\subseteq\mathbb{P}^2_{\overline{k}}$ has
degree $3d'$, hence it corresponds to the ideal $(3d')\subset \Z$
via the degree map. Consequently, the image of $C'$ in
$\operatorname{Pic}(B)$ is a generator of $d'\operatorname{Pic}(B)$.

\end{proof}

\section{The proof of Theorem \ref{thm-main}}

Let $B$ be the Brauer-Severi surface corresponding to $(L/k,\chi,a)$ as in Theorem \ref{thm-cyclic}. Then there is an isomorphism $\overline{\phi}:\,B\rightarrow\mathbb{P}^2$ defined over $k$ such that
$$
\xi(\sigma)=
         \begin{pmatrix}
           0 & 0 & a \\
           1 & 0 & 0 \\
           0 & 1 & 0
           \end{pmatrix}=\overline{\phi}\,.^{\sigma}\overline{\phi}^{-1},
$$
for $\sigma$ a generator of $\operatorname{Gal}(L/k)$ with
$\chi(\sigma)=-1$. The results in \cite{Lo2} apply to get the
equations in the statement of Theorem \ref{thm-main} for $B$ inside
$\mathbb{P}^9$. We recall that the equation are getting by twisting
the image of $\mathbb{P}^2$ into $\mathbb{P}^9$ by the Veronese
embedding
${\operatorname{Ver}_3}:\,\mathbb{P}^2\rightarrow\mathbb{P}^9$.
Indeed, we can compute following \cite{Lo2} \small
$$
{\operatorname{Ver}_3(\overline{\phi})}= \begin{pmatrix}
    a^2 l_1 & 0 & 0 & 0 & 0 & 0 & a^2 l_2 & 0 & 0 & a^2 l_3 \\
    0 & a l_1 & 0 & 0 & 0 & a l_2 & 0 & a l_3 & 0 & 0 \\
    0 & 0 & a l_1 & a l_2 & 0 & 0 & 0 & 0 & a l_3 & 0 \\
    0 & 0 & a l_2 & a l_3 & 0 & 0 & 0 & 0 & a l_1 & 0 \\
    0 & 0 & 0 & 0 & 1 & 0 & 0 & 0 & 0 & 0 \\
    0 & l_3 & 0 & 0 & 0 & l_1 & 0 & l_2 & 0 & 0 \\
    a l_2 & 0 & 0 & 0 & 0 & 0 & a l_3 & 0 & 0 & a l_1 \\
    0 & l_2 & 0 & 0 & 0 & l_3 & 0 & l_1 & 0 & 0 \\
    0 & 0 & l_3 & l_1 & 0 & 0 & 0 & 0 & l_2 & 0 \\
    l_3 & 0 & 0 & 0 & 0 & 0 & l_1 & 0 & 0 & l_2 \\
    \end{pmatrix}:\,B\rightarrow\mathbb{P}^2\subseteq\mathbb{P}^9,
$$
\normalsize
where $L=k(l_1,l_2,l_3)$ with
$\sigma(l_1)=l_2$ and $\sigma(l_2)=l_3$.

On the other hand, the twist $\phi:\,C'\rightarrow C^{d'}_{a}$ given
by the previous cocycle is embedded in $B$: we have the $k$-morphism
$f:\,C'\rightarrow B$ given by $\overline{\phi}^{-1}\Upsilon\phi$.
Composing with ${\operatorname{Ver}_3}$ we get the equations of $C'$
inside $\mathbb{P}^9$ in the statement of Theorem \ref{thm-main}.

Finally, the claim about the order of the curves $C'$ in
$\operatorname{Pic}(B)$ follows by Theorem \ref{prop-generator}.

\section{Generalizations on Picard group elements for {cyclic} Brauer-Severi varieties}
Let $L/k$ be a Galois cyclic extension of degree $n+1$, with
$\text{Gal}(L/k)=\langle\sigma\rangle$. Fix a character
$$\chi:\text{Gal}(L/k)\rightarrow\Z/(n+1)\Z,$$ which is equivalent to fix $\sigma'$ a generator of $Gal(L/k)$ such that $\chi(\sigma')=1$.
Given $a\in k^*$, we consider a $k$-algebra $(L/k,\chi,a)$ as
follows: As an additive group, $(L/k,\chi,a)$ is an
$n+1$-dimensional vector space over $L$ with basis $1,e,\ldots,e^n$:
$(L/k,\chi,a):=\oplus_{i=0}^{n}Le^i$ with $1=e^0$. Multiplication is
given by the relations: $e\,.\,\lambda=\sigma'(\lambda)\,.\,e$ for
$\lambda\in L$, and $e^{n+1}=a$. The algebra $(L/k,\chi,a)$ is
called the \emph{cyclic algebra} associated to the character $\chi$
and the element $a\in k$, and is trivial if and only if $a$ is a
norm of certain element of $L$. Its class in
$\operatorname{H}^1(k,\operatorname{PGL}_{n+1}(\overline{k}))$ corresponds to the
inflation of the cocycle in
$\operatorname{H}^1(\operatorname{Gal}(L/k),\operatorname{PGL}_{n+1}(L))$ given by
$$\xi(\sigma)=\left(\begin{array}{cccccc}
0&0&\ldots&\ldots&0&a\\
1&0&\ddots&\ddots&0&0\\
0&1&0&\ddots&\vdots&\vdots\\
\vdots&\ddots&\ddots&\ddots&&\vdots\\
0&0&\ldots&\ldots&1&0\\
\end{array} \right):=A_{\sigma}$$

%\begin{defn} A Brauer-Severi variety $B$ over $k$ of dimension $n$
%is called cyclic if the class $[B]$ in
%$\operatorname{H}^1(k,\operatorname{PGL}_{n+1}(\overline{k}))$ (by Corollary
%\ref{cor-Jah}) is also represented\footnote{{\color{red} Does any cyclic algebra is obtained in this way?}} by certain class of
%$(L/k,\chi,a)$ with $L/k$ a cyclic extension of degree $n+1$.
%\end{defn}

\begin{lem}\label{lemma-curves2} For any $a\in k^*$ and $d'\in\mathbb{Z}_{\geq 1}$, the
equation
$$\sum_{i=0}^{n} a^{i d'} X_i^{(n+1)d'}=0$$
defines {a non-singular $k$-projective model $X^{d',n}_a$ of degree
$(n+1)d'$ of a smooth projective variety inside $\mathbb{P}^{n}_k$},
such that $A_a:=[aX_n:X_0:\ldots:X_{n-1}]$ is leaving invariant
$X^{d',n}_a$.
\end{lem}
\begin{thm}Let $B$ be a Brauer-Severi variety over $k$, associated to a cyclic algebra
$(L/k,\chi,a)$ of dimension
${(n+1)^2}$ and exact order $n+1$ in $\operatorname{Br}(k)$.
%We assume that for any field $M$ with $k\subseteq
%M\subset L$, the algebra $(L/k,\chi,a)$ does not split in $M$.
For
any integer $d'\geq 1$, there is a twist $X'$ over $k$ of
$X^{d',n}_a$, living inside $B$ and defining a generator of $d'\operatorname{Pic}(B)$.
\end{thm}
\begin{proof}
Set
$m=\binom{2n+1}{n}$ and consider the Veronese embedding 
$\operatorname{Ver}_n:\,\mathbb{P}^n_k\hookrightarrow\mathbb{P}^{m-1}_k$. Use the $n$-Veronese embedding
$\operatorname{Ver}_n$ to
obtain a model of $X^{d',n}_a$ inside the trivial Brauer-Severi
variety $\operatorname{Ver}_n(\mathbb{P}^n_k)$. Because the image of a 1-cocyle
by the map
$\tilde{\operatorname{Ver}}_n:\operatorname{H}^1(k,\operatorname{PGL}_{n+1}(\overline{k}))\rightarrow
\operatorname{H}^1(k,\operatorname{PGL}_{m+1}(\overline{k}))$ is
equivalent to a 1-cocycle with coefficients in
$\operatorname{GL}_{m+1}(\overline{k})$ by \cite{Lo2} and
$\operatorname{H}^1(k,\operatorname{GL}_{m+1}(\overline{k}))=1$,
then $\tilde{\operatorname{Ver}}_n([B])$ is given in
$\operatorname{H}^1(k,\operatorname{PGL}_{m+1}(\overline{k}))$ by
$\tau\mapsto M\circ\,^
{\tau}M^{-1}$, for some $M\in\operatorname{GL}_{m+1}(\overline{k})$.
Consequently, $(M\circ\operatorname{Ver}_n)(\mathbb{P}^2_k)$ is a
model of $B$ in $\mathbb{P}^{m-1}_k$, containing
$(M\circ\operatorname{Ver}_n)(X^{d',n}_a)$ inside, which is a twist
of $X^{d',n}_a$ over $k$ associated to $\xi:\sigma\mapsto A_{\sigma}$.

On the other hand, by Theorem \ref{imprtiantexactsequ}, the map
$\delta$ sends $1$ to the Brauer class $[B]$ of $B$ inside the
$(n+1)$-torsion $\operatorname{Br}(k)[n+1]$ of the Brauer group
$\operatorname{Br}(k)$ of the field $k$. Hence $[B]$ has exact order
 $n+1$, being non-trivial, and so $\operatorname{Pic}(B)$ inside
$\operatorname{Pic}(B\otimes_k\overline{k}=\mathbb{P}^n_{\overline{k}}){\stackrel{\operatorname{deg}}{\cong}}\Z$
is isomorphic to $(n+1)\mathbb{Z}$. Moreover,
$X'\times_k\overline{k}\subseteq\mathbb{P}^n_{\overline{k}}$ has
degree $(n+1)d'$, hence it corresponds to the ideal
$((n+1)d')\subset \Z$ via the degree map. Consequently, the image of
$X'$ in $\operatorname{Pic}(B)$ is a generator of
$d'\operatorname{Pic}(B)$.

\end{proof}

Following the notation of \cite[Lemma 3.1]{Lo2}, we write $V_n:\,\mathbb{P}^n\rightarrow\mathbb{P}^m:\,(X_0:...:X_n)\mapsto(\omega_0:...:\omega_m)$, where the $\omega_k$ are equal to the products $\omega_{X_{0}^{\alpha_0}\dots X_{n}^{\alpha_n}}=\prod_{i}X_{i}^{\alpha_i}$ with $\sum_{i}\alpha_i=n+1$ in alphabetical order.

The automorphism $A_a$ of $X^{d',n}_a$ as an automorphism of $\operatorname{Ver}_n(X^{d',n}_a)$ sends $\omega_{X_{i}^n}\mapsto\omega_{X_{i+1}^n}$ and $\omega_{X_{n}^n}\mapsto a\omega_{X_{1}^n}$. 

\begin{cor} With notation above,
	$$
	\operatorname{Ver}_n(X'):\,\sum_{i=0}^{n}(\sum_j l_{i+j}\omega_{X_j^n})^{d'}=0\subseteq B
	$$
\end{cor}

\begin{proof}
	By using \cite[Section 3]{Lo}, we find that a matrix $\phi$ realizing the cocycle $\xi$, that is, $\xi=\phi\circ^{\sigma}\phi^{-1}$, sends $\phi(\omega_{X_i^n})=a^i(\sum l_{i+j}\omega_{X_j^n})$ and $\phi(\omega_{X_n^n})=(\sum l_{j-1}\omega_{X_j^n})$. We plug $\phi$ into the equation of $X^{d',n}_a$ and the result follows.
\end{proof}

\section*{Appendix: Another approach to construct Brauer-Severi
surfaces}
The third author shows an algorithm for constructing equations of Brauer-Severi varieties in \cite{Lo2}. Here we show an alternative way for constructing equations of Brauer-Severi surfaces ($n=2$) by using the Twisting Theory of plane curves.

Let $\operatorname{Ver}_{n}:\mathbb{P}^{n}_k\hookrightarrow\mathbb{P}^{\binom{2n-1}{n-1}-1}_k$ be the $n$-Veronese embedding.
It has been observed by the third author in \cite{Lo2} that the induced map
$$\tilde{\operatorname{Ver}}_{n}:\operatorname{H}^1(k,\operatorname{PGL}_{n+1}(\overline{k}))\rightarrow
\operatorname{H}^1(k,\operatorname{PGL}_{\binom{2n-1}{n-1}}(\overline{k})),$$
satisfies that the image of any 1-cocycle is equivalent to a
1-cocycle with values in the lineal group
$\operatorname{GL}_{\binom{2n-1}{n-1}}(\overline{k})$ and is
well-know that
%factors through
$\operatorname{H}^1(k,\operatorname{GL}_{\binom{2n-1}{n-1}}(\overline{k}))$,
is trivial by Hilbert 90 Theorem. This fact leads to an algorithm to
compute equations for any Brauer-Severi varieties. Here we use the idea coming from the construction in \cite{BaBaEl1} of the equations for a non-trivial Brauer-Severi surface.  

\begin{lem}\label{leminc} Let $C$ be a smooth plane curve over $k$ of genus $g=\frac{1}{2}(d-1)(d-2)\geq 3$.
The canonical embedding of $C$ is isomorphic to the composition
${\Psi}:C{\stackrel{\iota}{\longrightarrow}}\mathbb{P}^2_k{\stackrel{\operatorname{Ver}_{d-3}}{\longrightarrow}}\mathbb{P}^{g-1}_k,$
where {$\iota$ comes from the (unique) $g^2_d$-linear system,} all
are defined over $k$. In particular, fixing a non-singular plane
model $F_C(X,Y,Z)=0$ in $\mathbb{P}^2_k$ of $C$, one may directly
compute its canonical embedding into $\mathbb{P}^{g-1}_k$ by
applying the morphism $\operatorname{Ver}_{d-3}$.
\end{lem}
\begin{proof} It is fairly well-known
that the sheaves $\Omega^1(C)$ and $\mathcal{O}(d-3)|_{C}$ are
isomorphic (cf. R. Hartshorne \cite[Example 8.20.3]{Hart}). Hence,
$\operatorname{H}^0(\mathbb{P}^2,\mathcal{O}(d-3))\longrightarrow
\operatorname{H}^0(C,\Omega^1)$ is an isomorphism, and the statement
follows.
\end{proof}

Both maps, $\iota$ and $\operatorname{Ver}_{d-3},$ are
$\operatorname{Gal}(\overline{k}/k)$-equivariant. Therefore, the
natural maps
$$\operatorname{Aut}(\overline{C})\hookrightarrow
\operatorname{Aut}(\mathbb{P}^2_{\overline{k}})=\operatorname{PGL}_3(\overline{k})\rightarrow
\operatorname{Aut}(\mathbb{P}^{g-1}_{\overline{k}})=\operatorname{PGL}_{g}(\overline{k})$$
are morphisms of $\operatorname{Gal}(\overline{k}/k)$-groups.

\begin{prop}\label{prop-7.2} Given a non-trivial Brauer-Severi surface $B$ over $k$, associated to
a cyclic algebra $(L/k,\chi,a)$ of dimension 9, it is
algorithmically computable a $k$-model of $B$ in $\mathbb{P}^{\frac{9}{2}d'(d'-1)}_k$ with $d'\geq2$,
by considering any smooth plane curve $C$ over $k$ of degree $3d'$,
such that $[aZ:X:Y]$ is an automorphism. In particular, for $d'=2$ we get a model \footnote{In \cite{Lo2} the equations for Brauer-Severi surfaces are also obtained in $\mathbb{P}^9$.} in $\mathbb{P}^9$. 
\end{prop}

\begin{proof}
For non-hyperelliptic curves, see a description in \cite{Lo}, the
canonical model gives a natural
$\operatorname{Gal}(\overline{k}/k)$-inclusion
$\operatorname{Aut}(\overline{C})\hookrightarrow\operatorname{PGL}_g(\overline{k})$,
but we can go further, the action gives a
$\operatorname{Gal}(\overline{k}/k)$-inclusion
$\operatorname{Aut}(\overline{C})\hookrightarrow\operatorname{GL}_g(\overline{k})$.
In this way, the natural map
$\operatorname{H}^1(k,\operatorname{Aut}(\overline{C}))\rightarrow
\operatorname{H}^1(k,\operatorname{PGL}_{g}(\overline{k}))$,
satisfies that that the image of any 1-cocycle is equivalent to a
1-cocycle with values in $\operatorname{GL}_{g}(\overline{k})$, and
recall that
$\operatorname{H}^1(k,\operatorname{GL}_g(\overline{k}))$ is trivial
by applying Hilbert's Theorem 90. This allows us to compute
equations for twists via change of variables in
$\operatorname{GL}_g(\overline{k})$ of the canonical model for $C$.
Now, by Lemma \ref{leminc} and the proof of Theorem
\ref{prop-generator}, one could construct a smooth model for the
Brauer-Severi surface in $\mathbb{P}_k^9$ by taking the $V_{d-3}$
embedding with $d=6$ and $C$ as in the statement.
\end{proof}

%\section{Appendix}
%In order
%to apply Corollary \ref{cor15} or \ref{cor152}
% we need to construct a family of smooth plane curves over $k$ where
%$A_{\alpha_0}:=\left(\begin{array}{ccc}
%0&0&\alpha_0\\
%1&0&0\\
%0&1&0\\
%\end{array}\right)$ belongs to the automorphism group. We already constructed in degree 6 such candidates in
%\cite{BaBaEl1}, but we take a particular case for Picard curves of
%degree 6 which can be easily extended for any degree multiple of 3.
%And in order to consider also degree 3 smooth plane curves, we make
%the following modification of Corollaries
%\begin{rem} In \cite{BaBaEl1} we consider the set of curves in
%$\overline{\Q}$
%$$C_{\alpha_0,a}:\,Z^6+\frac{1}{\alpha_0^2}Y^6+\frac{1}{\alpha_0^4}X^6+\frac{a}{\alpha_0^3}(\alpha_0^2Z^3Y^3+\alpha_0X^3Z^3+Y^3X^3)=0,$$
%has the automorphism $[\alpha_0Z,X,Y]\in
%\Aut(\overline{C}_{\alpha_0,a})$ when is non-singular model, and in
%particular for $a\neq -10$,$\pm2$,$-1$,$0$
%${,\frac{1}{2}(-1\pm\sqrt{5})}$ and
%$\alpha_0\in\overline{\mathbb{Q}}$ one obtain that the automorphism
%group of the curves $C_{\alpha_0,a}$ is isomorphic to $\GAP(54,5)$
%which leaves inside a strata of the moduli space of non-singular
%plane curves of genus 10 by such group. All such curves can be also
%used in such section, but we are interested in any degree multiple
%of 3.
%\end{rem}
%
%


\begin{thebibliography}{10}
%   \bibitem{Es} E. Badr, \emph{On smooth plane curves}, PhD in
%   progress at UAB.
%   \bibitem{BaBa1} E. Badr and F. Bars, \emph{On the locus of smooth plane curves with a fixed automorphism
%       group}, Mediterr. J. Math. \textbf{13} (2016), 3605-3627. doi:\,10.1007/s00009-016-0705-9.
%
%   \bibitem{BaBa3}E. Badr and F. Bars, \emph{Automorphism groups of non-singular plane curves of degree 5 },
%   Commun. Algebra \textbf{44} (2016), 327-4340.
%   doi:\,10.1080/00927872.2015.1087547.
%
%   \bibitem{BaBa2}E. Badr and F. Bars, \emph{Non-singular plane curves with an element of ``large" order in its automorphism group }, Int. J. Algebra Comput. 26 (2016), 399-434. doi:\,
%   10.1142/S0218196716500168.

    \bibitem{BaBaEl1}E. Badr, F. Bars, and E. Lorenzo Garc\'ia; \emph{On twists of smooth plane curves}, arXiv:1603.08711.

%   \bibitem{BaEl1} E. Badr and E. Lorenzo, \emph{Parametrizing the moduli space of smooth plane curves of genus $6$}. In preperation.
%
%   \bibitem{Chang}H. C. Chang, \emph{On plane algebraic curves}, Chinese J. Math. \textbf{6} (1978), no. 2, 185-189.
%
%   \bibitem{Cha} F. Ch$\hat{a}$telet, \emph{Variations sur un th\`eme de H. Poincar\'e}, Ann. Sci. Ec. Norm. Sup.
%   61 (1944), 249-300.
%
%   \bibitem{GAP}  The GAP Group, GAP -Groups, Algorithms,  and Programming, Version 4.5.7, 2012. (http://www.gap-system.org)
%
%   \bibitem{Graaf} W. A. de Graaf, M. Harrison, J. Pilnikova, and J. Schicho, \emph{A Lie algebra method for
%       rational parametrization of Severi-Brauer surfaces}. arXiv:math/0501157v2 [math.AG]
%

\bibitem{GS} P. Gille, T. Szamuely, \emph{Central simple algebras and
Galois cohomology},  Cambridge
Studies in Advanced Mathematics 101, Cambridge University Press (2006).

\bibitem{Han} T. Hanke, \emph{The isomorphism problem for cyclic algebras and an application,} ISSAC 2007, ACM, New York, (2007), pp. 181–186. MR 2396201
\bibitem{Hart} R. Hartshorne, \emph{Algeraic Geometry,} Springer Verlag, (1978), New York.
%
%   \bibitem{Ha}T. Harui, \emph{Automorphism groups of plane
%       curves}, arXiv: 1306.5842v2[math.AG] 7 Jun 2014
%
%   \bibitem{Book} J. W. P. Hirschfeld, G. Korchm\'aros, and F. Torres,
%   \emph{Algebraic Curves over Finite Fields}, Princeton Series in
%   Applied Mathematics, 2008.
%
%   \bibitem{Hug} B. Huggins, \emph{Fields of moduli and fields of definition of
%       curves}. PhD thesis, Berkeley (2005), see
%   http://arxiv.org/abs/math/0610247v1.
%
   \bibitem{Ja} J. Jahnel, \emph{The Brauer-Severi variety associated with a
       central simple algebra: a survey}. See the book in
   https://www.math.uni-bielefeld.de/lag/man/052.pdf
%
%   \bibitem{LeRi} R. Lercier, C. Ritzenthaler, \emph{Hyperelliptic curves and
%       their invariants: geometric, arithmetic and algorithmic aspects}.
%   J. Algebra 372 (2012),
%   595-636.
%
%   \bibitem{LeRiRo} R. Lercier, C. Ritzenthaler, F. Rovetta, J. Sijsling, \emph{Parametrizing the moduli space of curves and applications to smooth plane quartics over finite fields}.
%   (LMS Journal of Computation and Mathematics, Volume 17, Special
%   Issue A (ANTS XI), LMS, London, pp. 128--147, 2014

%\bibitem{LomLor} D. Lombardo, E. Lorenzo Garc\'ia, \emph{Computing Twists of Hyperelliptic curves}, arxiv:1611.04856.

 % \bibitem{Loth} E. Lorenzo Garc\'ia, \emph{Arithmetic properties of non-hyperelliptic genus 3
 %      curves}. PhD dissertation, Universitat Polit\`ecnica de Catalunya
  % (2015), Barcelona.

  \bibitem{Lo} E. Lorenzo Garc\'ia, \emph{Twist of non-hyperelliptic curves},
   Rev. Mat. Iberoam. 33 (2017), no. 1, 169--182.

   \bibitem{Lo2} E. Lorenzo Garc\'ia, \emph{Construction of Brauer-Severi
   varieties}, preprint, arXiv:submit/1936164.
%
%   \bibitem{MT} S. Meagher, J. Top, \emph{Twists of genus three curves over finite fields}. Finite fields and their applications. 16, 5, p. 347-368, 2010.
%
%   \bibitem{Pa} R. Pannekoek, \emph{On the parametrization over $\mathbb{Q}$
%       of cubic surfaces}. Master dissertation, May 2009, University of
%   Groningen.
%
   \bibitem{RoXa} J. Ro\'e, X. Xarles, \emph{Galois descent for the gonality of curves},
   Arxiv:1405.5991v3 (2015).
%
 \bibitem{Se} J.P. Serre, \emph{Cohomologie Galoisianne}, LNM 5, Springer (1964).
%
   \bibitem{SeL} J.P. Serre, \emph{Local Fields}, GTM , Springer, (1980).
%
%   \bibitem{SuVo} J. Sijsling, J.Voight, \emph{On explicit descent of marked
%       curves and map}, arXiv:1504.02814v2 (2015).
%
   \bibitem{Ten} E. Tengan, \emph{Central Simple Algebras and the Brauer
       group} , XVII Latin American Algebra Colloquium, (2009). See the
   book in http://www.icmc.usp.br/~etengan/algebra/arquivos/cft.pdf
%
%   \bibitem{Wa} L. C. Washington, \emph{Introduction to cyclotomic fields}, GTM
%   83, Second edition, Springer, (1997).
%
   \bibitem{Wed} J. H. M. Wedderburn, \emph{On division algebras,} Trans. Amer. Math. Soc. \textbf{22} (1921),
no. 2, 129–135. MR 1501164
%
%   \bibitem{We} A. Weil, \emph{The field of definition of a variety}. American J. of Math. vol. 78, (1956),
%   509--524.
%
\end{thebibliography}
\end{document}